\documentclass[12pt,centertags,oneside]{amsart}
\usepackage[a4paper,margin=2.5cm]{geometry}
\usepackage[pagebackref]{hyperref}
\usepackage{amsmath,amssymb,amsfonts,amsthm}
\usepackage[abbrev]{amsrefs}
\usepackage[shortlabels]{enumitem}

\usepackage{mathrsfs,amsbsy,bm,mathtools,color}

\newtheorem{theorem}{Theorem}[section]

\newtheorem{prop}[theorem]{Proposition}

\newtheorem{lemma}[theorem]{Lemma}
\theoremstyle{definition}
\newtheorem{defn}[theorem]{Definition}
\newtheorem{definition}[theorem]{Definition}
\newtheorem{remark}[theorem]{Remark}

\newtheorem{example}[theorem]{Example}

\newcommand{\Q}{\mathbb{Q}}
\newcommand{\CC}{\mathbb{C}}
\newcommand{\R}{\mathbb{R}}
\newcommand{\Z}{\mathbb{Z}}

\newcommand{\lden}{\underline{\sf ud}}
\newcommand{\uden}{\overline{\sf ud}}

\newtheorem{cor}[theorem]{Corollary}

\renewcommand{\le}{\leqslant}
\renewcommand{\ge}{\geqslant}

\begin{document}

\title{Distribution modulo one of linear recurrent sequences}

\author{Zhangchi Chen} 
\address[Zhangchi Chen]{School of Mathematical Sciences, Key Laboratory of MEA (Ministry of Education) and Shanghai Key Laboratory of PMMP, East China Normal University, Shanghai 200241, China}
\email{zcchen@math.ecnu.edu.cn}

\author{Zihao Ye}
\address[Zihao Ye]{School of Mathematical Sciences, East China Normal University, Shanghai 200241, China}
\email{51265500041@stu.ecnu.edu.cn}

\author{Weizhe Zheng}
\address[Weizhe Zheng]{Morningside Center of Mathematics, Academy of Mathematics and Systems 
Science, Chinese Academy of Sciences, Beijing 100190, China; University of 
Chinese Academy of Sciences, Beijing 100049, China} \email{wzheng@math.ac.cn} 



\begin{abstract}
We study the distribution modulo one of linear recurrent sequences of real 
numbers. We prove criteria for the finiteness of the set of limit values of 
the fractional parts of such a sequence and give lower bounds for the 
maximal distance between two limit values. Our results generalize theorems 
of Flatto, Lagarias, Pollington, and Dubickas. 
\end{abstract}
\maketitle

\section{Introduction}
Distribution modulo one of sequences of real numbers has been intensively 
studied since Weyl \cite{Weyl}. Among other results, Weyl showed that for any 
polynomial $F(x)\in \R[x]$ such that $F(x)-F(0)\notin \Q[x]$, the sequence 
$(F(k))_{k\ge 1}$ is uniformly distributed modulo one \cite[Satz~12]{Weyl}. 
In another direction, the distribution of $(\{\xi \alpha^k\})_{k\ge 1}$ for 
$\xi,\alpha\in \R$ was studied by Hardy \cite{Hardy}, Pisot \cite{Pisot}, 
Vijayaraghavan \cite{Vija}, and many others. Here $\{x\}:=x-\lfloor x\rfloor 
$ denotes the fractional part of a real number $x$. For $\alpha=3/2$, Mahler 
\cite{Mahler} considered the following question: Does there exist a real 
number $\xi>0$ such that $(\{\xi (3/2)^k\})_{k\ge 1}$ lies in the interval 
$[0,1/2)$? While this question remains open, progress was made by Flatto, 
Lagarias, and Pollington \cite[Theorem 1.4]{FLP}, who showed that for any 
real number $\xi>0$ and coprime integers $p>q> 1$, we have 
\begin{equation}\label{eq:FLP}
\limsup_{k\to \infty}\{\xi(p/q)^k\}
-\liminf_{k\to \infty}\{\xi(p/q)^k\}\ge 1/p.
\end{equation}
Dubickas (\cite{Dubickas-BLMS}, \cite{Dubickas-2007}) extended this result to 
certain sequences of the form $(\{F(k)\alpha^k\})_{k\ge 1}$, where $\alpha>1$ 
is a real algebraic number. 

In this paper, we study more generally distribution modulo one of linear 
recurrent sequences. One special case of our main result is the following 
extension of the above-mentioned theorem of Flatto, Lagarias, and Pollington. 

\begin{theorem}\label{t:rational}
Let $\xi_1,\dots,\xi_n$ be nonzero real numbers and let 
$p_1,\dots,p_n,q_1,\dots,q_n$ be integers such that $p_i>\lvert q_i\rvert 
>0$ for all $1\le i\le n$ and
$\frac{p_1}{q_1},\dots,\frac{p_n}{q_n}$ are pairwise distinct. Assume that 
either $\xi_1$ is irrational or $p_1/q_1\notin\Z$. Then, for every integer 
$M\ge 1$, there exists an integer $l$ such that 
\[
\limsup_{k\to \infty}\left\{\sum_{i=1}^n\xi_i\left(\frac{p_i}{q_i}\right)^{kM+l}\right\}
-\liminf_{k\to \infty}\left\{\sum_{i=1}^n\xi_i\left(\frac{p_i}{q_i}\right)^{kM+l}\right\}\ge \frac{1}{p_1\dotsm p_n}.
\]
\end{theorem}

Taking $n=M=1$, we recover \eqref{eq:FLP}. For coprime integers $p> q> 1$, 
let $Z_{p/q}(s,t)$ be the set considered in \cite{FLP} consisting of real 
numbers $\xi>0$ such that $\{\xi(p/q)^k\}\in [s,t)$ for all $k\ge 0$. The 
case $n=1$ of Theorem \ref{t:rational} implies that, for every integer $M\ge 
1$, at least one of the real numbers $\xi$, $\xi(p/q)$, \dots, 
$\xi(p/q)^{M-1}$ does not belong to $Z_{(p/q)^M}(s,t)$ as long as $t-s<1/p$. 

Before stating our main theorem, we fix some notation and terminology. Let 
$\pi\colon \R\to \R/\Z$ denote the projection. For any real number $x$, we 
let $\lVert x\rVert$ denote the distance from $x$ to the nearest integer. By 
a \emph{linear recurrent sequence} we mean a sequence of real numbers 
$(x_k)_{k\ge 1}$ such that there exist $a_0,\dots,a_d\in \Z$ with $a_d\neq 0$ 
satisfying $\sum_{i=0}^d a_ix_{k+i}=0$ for all $k\ge 1$. Note that we allow 
$a_0=0$ and do not require $a_d= 1$. For any polynomial $P(x)=\sum_{i=0}^d 
a_i x^i\in \CC[x]$, its \emph{length} is defined to be $L(P):=\sum_{i=0}^d 
\lvert a_i\rvert$. For $P(x)\in \R[x]$, Dubickas defined its \emph{reduced 
length} to be $\ell(P):=\inf_Q L(PQ)$, where $Q$ runs through polynomials in 
$\R[x]$ of leading coefficient~$1$ or constant coefficient~$1$. We introduce 
a variant. 

\begin{defn}\label{d:reduced}
The \emph{overreduced length} of $R\in \Z[x]$ is defined to be 
$\lambda(R):=\min_Q \ell(R/Q)$, where $Q$ runs through factors of $R$ in 
$\Z[x]$ such that every root of $Q$ is either a root of unity or has modulus 
$< 1$ and every root of $Q$ that is a root of unity is simple. 
\end{defn}

By definition $\lambda(R)\le \ell(R)$. 

We can now state the main theorem of this paper.

\begin{theorem}\label{t:Dubickas}
Let $R(x)=\sum_{i=0}^r u_i x^i\in \Z[x]$ be a nonzero polynomial and let 
$(x_k)_{k\ge 1}$ be a sequence of real numbers satisfying the linear 
recurrence relation $\sum_{i=0}^r u_ix_{k+i}=0$ for all $k\ge 1$. Assume 
$x_k=\sum_{i=1}^n F_i(k)\alpha_i^k$ for all $k\ge 1$, where 
$\alpha_1,\dots,\alpha_n\in \CC$ are pairwise distinct and $F_1,\dots,F_n\in 
\CC[x]$. Assume that at least one of the following conditions holds: 
\begin{enumerate}[(a)]
\item There exists $1\le i\le n$ such that $F_i\neq 0$ and 
    $\lvert\alpha_i\rvert>1$ and $\alpha_i$ is not an algebraic integer; or 
\item There exists $1\le i\le n$ such that $\deg(F_i)\ge 1$ and 
    $\lvert\alpha_i\rvert=1$ and $\alpha_i$ is not an algebraic integer; or 
\item There exist $1\le i,j\le n$ and a field automorphism $\sigma\colon 
    \CC\to \CC$ such that $\lvert \alpha_i\rvert>1$, $\lvert 
    \alpha_j\rvert>1$, $\sigma(\alpha_i)=\alpha_j$, and $\sigma (F_i)\neq 
    F_j$; or 
\item There exist $1\le i,j\le n$ and a field automorphism $\sigma\colon 
    \CC\to \CC$ such that $\lvert \alpha_i\rvert\ge 1$, $\lvert 
    \alpha_j\rvert\ge 1$, $\sigma(\alpha_i)=\alpha_j$, and $\deg(\sigma 
    (F_i)- F_j)\ge 1$. 
\end{enumerate} 
Then the set $E$ of limit values of $(\pi(x_{k}))_{k\ge 1}$ is an infinite 
set and 
\begin{equation}\label{eq:Dubickas2}
\limsup_{k\to \infty} \lVert x_k\rVert \ge 1/L(R)
\end{equation} 
and, for every integer $M\ge 1$, there exists an integer $l\ge 0$ such that 
the set of limit values of $(\pi(x_{kM+l}))_{k\ge 1}$ is not contained in any 
interval of $\R/\Z$ of length $<1/\lambda(R)$. Moreover, the assertion that 
$E$ is an infinite set still holds under the following condition: 
\begin{enumerate}
\item[(e)] There exist $1\le i,j\le n$ and a field automorphism 
    $\sigma\colon \CC\to \CC$ such that $\lvert \alpha_i\rvert> 1$, $\lvert 
    \alpha_j\rvert\ge 1$, $\sigma(\alpha_i)=\alpha_j$, and $\sigma 
    (F_i)\neq F_j$. 
\end{enumerate}
\end{theorem}

By taking $\sigma(\alpha_i)=\alpha_i$, we see that condition (c) is satisfied 
in the following case: 
\begin{enumerate}
\item[(c$'$)] There exists $1\le i\le n$ such that $\lvert 
    \alpha_i\rvert>1$ and $F_i\notin \Q(\alpha_i)[x]$. 
\end{enumerate} 
Similarly, condition (d) is satisfied in the following case:
\begin{enumerate}
\item[(d$'$)] There exists $1\le i\le n$ such that $\lvert 
    \alpha_i\rvert\ge 1$ and $F_i(x)-F_i(0)\notin \Q(\alpha_i)[x]$. 
\end{enumerate} 

\begin{remark}
\begin{enumerate}
\item There exist $\alpha_1,\dots,\alpha_n\in \CC$ pairwise distinct and 
    $F_1,\dots,F_n\in \CC[x]$ satisfying one of the conditions (a), (b), 
    (c), (d) such that $x_k=\sum_{i=1}^n F_i(k)\alpha_i^k$ for all $k\ge 1$ 
    if and only if $(x_k)_{k\ge 1}$ is not a finite sum of sequences (of 
    complex numbers) of the following types: 
\begin{enumerate}[(A)]
\item $(\sum_\sigma \sigma (F(k)\alpha^k))_{k\ge 1}$, where $\alpha\in \CC$ 
    is an algebraic integer, $F(x)\in \Q(\alpha)[x]$, and $\sigma$ runs 
    through embeddings $\Q(\alpha)\to \CC$; 
\item $(\xi\alpha^k)_{k\ge 1}$, where $\xi\in \CC$ and $\lvert 
    \alpha\rvert=1$; 
\item $(F(k)\alpha^k)_{k\ge 1}$, where $F(x)\in \CC[x]$ and $\lvert 
    \alpha\rvert <1$. 
\end{enumerate}

\item The assertion on $\lambda(R)$ in Theorem \ref{t:Dubickas} implies 
    that, for every integer $M\ge 1$, there exists an integer $l$ such that 
\begin{gather*}
\limsup_{k\to \infty}\{x_{kM+l}\}- \liminf_{k\to \infty} \{x_{kM+l}\}\ge 1/\lambda(R),\\
\limsup_{k,k'\to \infty} \lVert  x_{kM+l}- x_{k'M+l}\rVert\ge 1/\max(3,\lambda(R)).
\end{gather*}
\end{enumerate}
\end{remark}

Next we look at some special cases of the theorem. The following 
characterization of the finiteness of the set of limit values generalizes 
classical theorems of Vijayaraghavan \cite[Theorem~4]{Vija} (sufficient 
condition in the case $\deg(F_i)=0$ for all $i$) and Pisot 
\cite[Th\'eor\`eme~II]{Pisot-1946} (case $n=1$, $\deg(F_1)=0$). 

\begin{cor}\label{c:iff}
Let $(x_k)_{k\ge 1}$ be a sequence of real numbers such that 
$x_k=\sum_{i=1}^n F_i(k)\alpha_i^k$ for all $k\ge 1$, where $F_1,\dots,F_n\in 
\CC[x]$ are nonzero and $\alpha_1,\dots,\alpha_n\in \CC$ are pairwise 
distinct algebraic numbers. Assume that for each $1\le i\le n$, $\alpha_i$ is 
either a root of unity or Galois conjugate to a complex number of modulus 
$\neq 1$. Then the set of limit values of $(\{x_k\})_{k\ge 1}$ is finite if 
and only if all the following conditions are satisfied: 
\begin{enumerate}[(1)]
\item For every $1\le i\le n$ such that $\lvert \alpha_i\rvert > 1$, 
    $\alpha_i$ is an algebraic integer; and 
\item For every field automorphism $\sigma\colon \CC\to \CC$ and every 
    $1\le i\le n$ such that $\lvert \alpha_i\rvert>1$, either $\lvert 
    \sigma(\alpha_i)\rvert <1$ or there exists $1\le j\le n$ such that 
    $\sigma(\alpha_i)=\alpha_j$ and $\sigma(F_i)=F_j$; and 
\item For every field automorphism $\sigma\colon \CC\to \CC$ and every 
    $1\le i\le n$ such that $\alpha_i$ is a root of unity and $\deg(F_i)\ge 
    1$, there exists $1\le j\le n$ such that $\sigma(\alpha_i)=\alpha_j$ 
    and $\deg(\sigma(F_i)-F_j)\le 0$.   
\end{enumerate}
\end{cor} 

The following two corollaries are consequences of Theorem \ref{t:Dubickas} 
and Proposition \ref{prop} in the case $x_k=F(k)\alpha^k +G(k)$, where 
$F(x),G(x)\in \R[x]$ are polynomials. Recall that a \emph{Pisot number} is a 
real algebraic integer $\alpha>1$ all of whose Galois conjugates other than 
$\alpha$ have modulus strictly less than $1$, and a \emph{Salem number} is a 
real algebraic integer $\alpha>1$ all of whose Galois conjugates other than 
$\alpha$ have modulus $\le 1$ such that $\alpha$ is not a Pisot number. 
Although we do not need this fact, recall also that a real algebraic number 
$\alpha>1$ is a Pisot number if and only if there exists a real number 
$\xi\neq 0$ such that $\lVert \xi \alpha^k\rVert \to 0$ as $k\to \infty$, by 
a theorem of Hardy \cite[Theorem~A]{Hardy} (see also \cite{Pisot}). 

\begin{cor}\label{c:1}
Let $\alpha$ be a real algebraic number with $\lvert\alpha\rvert>1$ and let 
$F(x),G(x)\in \R[x]$ with $F(x)$ nonzero. Then the set $E$ of limit values of 
$(\pi(F(k)\alpha^k+G(k)))_{k\ge 1}$ is finite if and only if $\lvert 
\alpha\rvert$ is a Pisot number, $F(x)\in \Q(\alpha)[x]$, and $G(x)-G(0)\in 
\Q[x]$. Moreover, if $G(x)-G(0)\notin \Q[x]$, then $E$ is uncountable. 
\end{cor}

The case $g\le 0$ (and $\alpha>1$) of the first assertion of the following 
corollary recovers theorems of Dubickas: The case $f=0\ge g$  contains
\cite[Theorem~2]{Dubickas} and the case $f>0> g$ contains 
\cite[Theorem]{Dubickas-2007}. We adopt the convention $\deg(0)=-1$.

\begin{cor}\label{c:2}
Let $\alpha$ be a real algebraic number with $\lvert \alpha\rvert >1$ and let 
$F(x),G(x)\in \R[x]$ with $f=\deg(F) \ge 0$ and $g=\deg(G)$. Let $P(x)\in 
\Z[x]$ be the minimal polynomial of $\alpha$. Assume that at least one of the 
following conditions holds: 
\begin{enumerate}[(1)]
\item $\lvert \alpha\rvert$ is neither a Pisot number nor a Salem number; 
    or 
\item $F(x)\notin \Q(\alpha)[x]$; or 
\item $G(x)-G(0)\notin \Q[x]$; or
\item $\lvert \alpha\rvert$ is a Salem number and $f>0$.
\end{enumerate}
Let $x_k=F(k)\alpha^k+G(k)$ for $k\ge 1$. Then
\begin{equation}\label{eq:c2} 
\limsup_{k\to \infty}\lVert x_{k}\rVert\ge 1/L((x-1)^{g+1}P(x)^{f+1}Q(x)),
\end{equation}
for every nonzero $Q\in \Z[x]$, and, for every integer $M\ge 1$, there exists 
an integer $l\ge 0$ such that the set of limit values of 
$(\pi(x_{kM+l}))_{k\ge 1}$ is not contained in any interval of $\R/\Z$ of 
length $<1/\ell((x-1)^{\max(g,0)}P(x)^{f+1})$. Moreover, under assumption 
(3), the set of limit values of $(\pi(x_k))_{k\ge 1}$ is not contained in any 
interval of length $<1/\inf_{Q}L(P^{f+1}Q)$, where $Q$ runs through 
polynomials in $\Z[x]$ satisfying $Q(1)\neq 0$. 
\end{cor}

We prove Theorem \ref{t:Dubickas} by refining the techniques developed by 
Dubickas in \cite{Dubickas-BLMS}, \cite{Dubickas}, \cite{Dubickas-2007}. One 
crucial observation by Dubickas and Novikas \cite[Lemma~2]{DN} is that the 
sequence $(\lfloor \xi\alpha^{k+1}\rfloor-\alpha\lfloor 
\xi\alpha^{k}\rfloor)_{k\ge 1}$ is not ultimately periodic for $\xi>0$ and 
$\alpha>1$ satisfying $\alpha\in \Q\backslash \Z$. This was extended by 
Dubickas to certain sequences associated to $(\lfloor F(k) 
\alpha^k\rfloor)_{k\ge 1}$, where $\alpha>1$ is a real algebraic number.  We 
show that if one of the conditions (a), (b), (c), (d) in Theorem 
\ref{t:Dubickas} holds, then $(\lfloor x_k\rfloor)_{k\ge 1}$ is not a linear 
recurrent sequence (see Proposition \ref{p:recur} for a more general 
statement), which implies the bounds in Theorem \ref{t:Dubickas} by the 
following general result. 

\begin{theorem}\label{p:Dubickas}
Let $R(x)=\sum_{i=0}^r u_i x^i\in \Z[x]$ be nonzero and let $(x_k)_{k\ge 1}$ 
be a sequence of real numbers satisfying the linear recurrence relation 
$\sum_{i=0}^r u_i x_{k+i}=0$ for all $k\ge 1$. Assume that $x_k=y_k+z_k$ with 
$z_k\in \Z$ for all $k\ge 1$ such that $(y_k)_{k\ge 1}$ is bounded and 
$(z_k)_{k\ge 1}$ is not a linear recurrent sequence. Then
\[\limsup_{k\to \infty} \lvert y_k\rvert\ge 1/L(R)\] 
and for any integer $M\ge 1$, there exists an integer $l$ such that 
\begin{equation}\label{eq:pD1}
\limsup_{k\to \infty} y_{kM+l} - \liminf_{k\to \infty} y_{kM+l}\ge 1/\lambda(R).
\end{equation}
\end{theorem}

Our original motivation for studying the distribution modulo one of linear 
recurrent sequences comes from the work of Astorg and Boc Thaler 
\cite{Astorg-Boc Thaler-2024} in complex dynamics. To each increasing 
sequence $(n_k)_{k\ge 1}$ of positive integers, they associated a phase 
sequence $(n_{k+1}-\alpha n_k-\beta \ln n_k)_{k\ge 1}$, where $\alpha>1$ and 
$\beta$ are real numbers. Taking $n_k=\lfloor\xi\alpha^k-k\theta\rfloor$ with 
$\xi>0$ and $\theta=\frac{\beta\ln \alpha}{\alpha-1}$, we have 
\[n_{k+1}-\alpha n_k-\beta \ln n_k=-\{\xi\alpha^{k+1}-(k+1)\theta\}+\alpha\{\xi\alpha^k-k\theta\}-\theta-\beta\ln\xi+o(1).\]
Thus our previous results in \cite{CYZ} on the distribution of the phase 
sequence imply restrictions on the distribution of 
$(\{\xi\alpha^k-k\theta\})_{k\ge 1}$. The results obtained in this paper by 
working directly with linear recurrent sequences turned out to be stronger. 

This paper is organized as follows. In Section~\ref{s:Weyl}, we use a theorem 
of Hlawka and Lawton on well-distribution extending Weyl's theorem to study 
one class of linear recurrent sequences (Proposition \ref{prop}). In 
Section~\ref{s:proofs}, we prove all the results stated in the introduction. 
In Section~\ref{s:final}, we briefly discuss variants of the notion of 
spectrum of a sequence.

\section{A consequence of well-distribution}\label{s:Weyl}

\begin{defn}
Let $(a_0,\dots,a_d)$ be a finite collection of integers. We say that a 
subset $I\subseteq \R/\Z$ is \emph{$(a_0,\dots,a_d)$-admissible} if the 
Minkowski sum 
\[\sum_{s=0}^d a_s I=\{\sum_{s=0}^da_s x_s\mid x_0,\dots,x_d\in 
I\}
\] 
is \emph{not} dense in $\R/\Z$. 
\end{defn}

\begin{example}\label{ex:adm}
\begin{enumerate}
\item Any countable closed subset of $\R/\Z$ is 
    $(a_0,\dots,a_d)$-admissible.
\item Any interval of $\R/\Z$ of length $<1/\sum_{s=0}^d \lvert a_s\rvert$ 
    is $(a_0,\dots,a_d)$-admissible. 
\end{enumerate}
\end{example}

For any subset $A\subseteq \mathbb{Z}_{\geqslant 1}$, recall that its {\em 
upper uniform density} and {\em lower uniform density} are defined to be 
\begin{align*}
\uden(A)&:=\lim_{n\to \infty}\frac{1}{n}\sup_{h\in \Z_{\ge 0}}\#(A\cap [h+1,h+n]),\\
\lden(A)&:=\lim_{n\to \infty}\frac{1}{n}\inf_{h\in \Z_{\ge 0}}\#(A\cap [h+1,h+n]).
\end{align*}
We refer to \cite{GTT} for a proof that the limits exist and for equivalent 
definitions. We say that $A$ has \emph{uniform density} $\delta$ if  
$\uden(A)=\lden(A)=\delta$. 

Recall that $\pi\colon \R\to \R/\Z$ denotes the projection. 

\begin{prop}\label{prop}
Let $P(x)=\sum_{s=0}^d a_s x^s\in \Z[x]\setminus\{0\}$, $F(x)=\sum_{i=0}^f 
\theta_i x^i\in \R[x]$. Let $(x_k)_{k\ge 1}$ be a sequence of real numbers 
satisfying the linear recurrence relation $\sum_{s=0}^d a_sx_{k+s}=0$ for all 
$k\ge 1$. Assume that there exists an $(a_0,\dots,a_d)$-admissible subset 
$I\subseteq \R/\Z$ such that for any open neighborhood $J$ of $I$, $\{k\ge 
1\mid \pi(x_k+F(k)) \in J\}$ has upper uniform density $1$. Then $\theta_i\in 
\Q$ for all $\nu< i\le f$, where $\nu\geqslant 0$ is the order of $P(x)$ at 
$x=1$. In particular, if $P(1)\neq 0$, then $F(x)-F(0)\in \Q[x]$. 
\end{prop}

\begin{remark}\label{r:Spec}
The assumption of Proposition~\ref{prop} is satisfied if the set of limit 
values of $(\pi(x_k+F(k)))_{k\ge 1}$ is $(a_0,\dots,a_d)$-admissible. 
\end{remark} 

A sequence $(x_k)_{k\ge 1}$ in $\R/\Z$ is said to be \emph{well distributed} 
if for every interval $I\subseteq \R/\Z$, the set $\{k\ge 1\mid x_k\in I\}$ 
has uniform density equal to the length of $I$. This property was introduced 
by Hlawka \cite{Hlawka0} and Petersen \cite{Petersen}. The proof of 
Proposition \ref{prop} uses the following extension by Hlawka \cite{Hlawka} 
and Lawton \cite[Theorem~2]{Lawton} of Weyl's theorem \cite[Satz 12]{Weyl}: 
For any polynomial $F(x)\in \R[x]$ such that $F(x)-F(0)\notin \Q[x]$, the 
sequence $(\pi(F(k)))_{k\ge 1}$ is well distributed. 

In general lower uniform density fails to be subadditive: There exist subsets 
$A,B\subseteq \Z_{\ge 1}$ such that $\lden(A\cup B)>\lden(A)+ \lden(B)$. The 
following lemma for the lower uniform density of a finite union of translates 
appeared already in \cite[Lemma 3.9]{CYZ}. We include a proof here for the 
sake of completeness. For $A\subseteq \Z_{\ge 1}$ and $m\in \Z$, we put 
$A[m]=\{k\ge 1\mid k+m\in A\}$. 

\begin{lemma}\label{l:upper-den}
Let $A\subseteq\mathbb{Z}_{\geqslant1}$ and let $M$ be a finite set of 
integers. Then $\lden(\bigcup_{m\in M}A[m]) \le \#M\cdot \lden(A)$. 
\end{lemma}

\begin{proof}
We may assume that $M$ is nonempty. Let $a_n=\#( A\cap [1,n])$. For $n\le 0$ 
we adopt the convention that $a_n=0$. Then $a_n\le a_{n+1}\le a_n+1$ for all 
$n\in \Z$. Let $U=\bigcup_{m\in M}A[m]$ and $k=\max_{m\in M} \lvert m\rvert$. 
Then 
\begin{multline*}
\#(U\cap 
[h+1,h+n])\le \sum_{m\in M}\#(A[m]\cap[h+1,h+n])=\sum_{m\in M}(a_{h+n+m}-a_{h+m})\\\le \#M\cdot (a_{h+n+k}-a_{h-k})\le \#M\cdot 
(a_{h+n}-a_h+2k)=\#M\cdot (a_{h+n}-a_h) +O(1).
\end{multline*}
Thus 
\[\lden(U)=\lim_{n\to \infty} \frac{1}{n}\inf_{h\ge 0}\#(U\cap 
[h+1,h+n])\le\#M\cdot \lim_{n\to \infty}\frac{1}{n}\inf_{h\ge 0} (a_{h+n}-a_h)=\#M\cdot \lden(A).\qedhere
\] 
\end{proof}

\begin{proof}[Proof of Proposition \ref{prop}]
Let $\gamma_k=x_k+F(k)$ and let $\tilde\gamma_k=\sum_{s=0}^d 
a_s\gamma_{k+s}$. Then 
\[
\tilde{\gamma}_k
=\underbrace{\left(\sum\limits_{s=0}^d a_sx_{k+s}\right)}_{=0}
+\sum\limits_{s=0}^d\sum_{i=0}^f a_s\theta_i(k+s)^i
=\sum\limits_{j=0}^f c_jk^j,
\]
where 
\[
c_j
=\sum_{s=0}^d a_s\sum_{i=j}^f \binom{i}{j}\theta_i s^{i-j}
=\sum_{i=j}^f \binom{i}{j}\theta_i(\Theta^{i-j}P)(1).
\] 
Here $\Theta=x\frac{d}{dx}$ and we adopt the usual convention $s^0=1$ (even 
for $s=0$). Since $\nu\ge 0$ is the order of $P(x)$ at $x=1$, $(\Theta^t 
P)(1)=0$ for all $0\le t < \nu$ and $(\Theta^{\nu} P)(1)=P^{(\nu)}(1)\neq 0$. 
Assume that $\theta_j$ is irrational for some $j>\nu$. We may assume that $j$ 
is the largest such integer. Then, 
\[c_{j-\nu}=\binom{j}{j-\nu}\theta_j 
P^{(\nu)}(1)+\sum_{i=j+1}^{f} 
\binom{i}{j-\nu}\theta_i(\Theta^{i-j+\nu}P)(1)\] is irrational, so that 
$(\pi(\tilde \gamma_k))_{k\ge 1}$ is well distributed by the theorem of 
Hlawka and Lawton recalled earlier. By the admissibility assumption, the 
complement of $\sum_{s=0}^d a_s I$ has nonempty interior, and thus contains a 
closed interval $K$ of positive length. By continuity, there exists an open 
neighborhood $J$ of $I$ such that $\sum_{s=0}^d a_s J$ is contained in the 
complement of $K$. By assumption and Lemma \ref{l:upper-den}, 
\[\{k\ge 1\mid \pi(\tilde \gamma_k)\in K\}\subseteq 
\bigcup_{s=0}^d \{k\ge 1\mid \pi(\gamma_{k+s})\notin J\}
\] 
has lower uniform density $0$, which contradicts the well-distribution of 
$(\pi(\tilde \gamma_k))_{k\ge 1}$. 
\end{proof}

\section{Proofs of the main results}\label{s:proofs}

\begin{prop}\label{p:recur}
Notation as in Theorem \ref{t:Dubickas}. Assume $x_k=y_k+z_k$ with $z_k\in 
\Z$ for all $k\ge 1$ and $(y_k)_{k\ge 1}$ bounded. If one of the conditions 
(a), (b), (c), or (d) holds, then $(z_k)_{k\ge 1}$ is not a linear recurrent 
sequence. Moreover, there exists a constant $N>0$ depending only on $R$ (and 
not on $(x_k)_{k\ge 1}$), such that, whenever there exists $k_0\ge 1$ such 
that $\sum_{m=0}^r u_m z_{k+m}=0$ for all $k\ge k_0$, we have $\limsup_{k\to 
\infty} \lvert y_k\rvert\ge N\cdot \lvert \sigma(F_i)-F_j\rvert$ under 
condition (e). 
\end{prop} 

Note that in the last assertion of Proposition \ref{p:recur}, we necessarily 
have $\sigma(F_i)-F_j\in \CC$. Indeed, case (d) cannot hold by the first 
assertion of the proposition. 

\begin{proof}
We refine the arguments by Dubickas in the first part of 
\cite[Section~4]{Dubickas-2007} using techniques of the proof of Proposition 
\ref{prop}. Up to reordering and adding zeroes to the sequence 
$F_1,\dots,F_n$ if necessary, we may assume that $i=1$ in the conditions (a), 
(b), (c), (d), and (e), and moreover $\alpha_1,\dots,\alpha_d$ are the Galois 
conjugates of $\alpha_1$. Suppose that $(z_k)_{k\ge 1}$ is a linear recurrent 
sequence. 

For the first assertion of Proposition \ref{p:recur}, we may assume, by 
replacing $R$ with a multiple of $R$, that the sequence $(z_k)_{k\ge 1}$ 
satisfies $\sum_{i=0}^r u_i z_{k+i}=0$ for all $k\ge 1$.  Let $P(x)\in \Z[x]$ 
be the minimal polynomial of $\alpha_1$, of degree $d$. Write $R=P^{f+1} Q$, 
where $P$ and $Q\in \Z[x]$ are coprime. Write $P^{f+1}(x)=\sum_{i=0}^{d(f+1)} 
v_i x^i$, $Q(x)=\sum_{i=0}^{e}b_ix^i$, where $d=\deg(P)$ and $e=\deg(Q)$. Let 
\[\tilde x_k=\sum_{i=0}^e b_i x_{k+i},\quad \tilde y_k=\sum_{i=0}^e b_i y_{k+i},\quad \tilde z_k=\sum_{i=0}^e 
b_iz_{k+i}.
\] 
Then $\tilde z_k$ satisfies the linear recurrence relation $\sum_{i=0}^{d(f+1)} v_i 
\tilde z_{k+i}=0$ for all $k\ge 1$. Thus there exist $G_1,\dots,G_d\in 
\CC[x]$ such that $\tilde z_k=\sum_{j=1}^d G_j(k)\alpha_j^k$ for all $k\ge 
1$. By \cite[Lemma~1]{Dubickas-2007}, if $G_1,\dots,G_d$ are not all zero, 
then, for all $1\le j\le d$, $\alpha_j$ is an algebraic integer and 
$G_j(x)\in \Q(\alpha_j)[x]$, and every isomorphism $\sigma\colon 
\Q(\alpha_i)\to \Q(\alpha_j)$ sending $\alpha_i$ to $\alpha_j$ sends $G_i$ to 
$G_j$ for all $1\le i,j\le d$. 

For $1\le j\le d$, write $F_j(x)=\sum_{i=0}^{f_j}\xi_{j,i}x^i$, where 
$f_j=\deg(F_j)$. We have 
\begin{equation}\label{eq:temp}
\tilde x_k=\sum_{s=0}^e b_s (\sum_{j=1}^d F_j(k+s)\alpha_j^{k+s})
=\sum_{s=0}^e\sum_{j=1}^d\sum_{i=0}^{f_j} b_s\xi_{j,i}(k+s)^i\alpha_j^{k+s}=\sum_{j=1}^d\alpha_j^k\sum_{g=0}^{f_j} c_{j,g} k^g,
\end{equation}
where 
\begin{equation}\label{eq:cj}
c_{j,g}
=\sum_{s=0}^e b_s\sum_{i=g}^{f_j} \binom{i}{g}\xi_{j,i} s^{i-g}\alpha_j^s
=\sum_{i=g}^{f_j} \binom{i}{g}\xi_{j,i}(\Theta^{i-g}Q)(\alpha_j).
\end{equation}
Here $\Theta=x\frac{d}{dx}$. Note that $Q(\alpha_j)\neq 0$ for all $1\le j\le 
d$. Let $H_j(x)=\sum_{i=0}^{f_j} c_{j,i} x^i$. From \eqref{eq:temp}, we get 
\begin{equation}\label{eq:yk}
\tilde y_k=\tilde x_k-\tilde z_k = \sum_{j=1}^d (H_j(k)-G_j(k))\alpha_j^k.
\end{equation}

Consider the power series
\begin{equation}\label{eq:ps} 
\sum_{k=1}^\infty \tilde y_kx^k=\sum_{j=1}^d\sum_{k=1}^\infty (H_j(k)-G_j(k))\alpha_j^kx^k.
\end{equation}
Note that, for any nonzero polynomial $S(x)\in \CC[x]$ and $\beta\in 
\CC\backslash \{0\}$, $\sum_{k=1}^\infty S(k)\beta^kx^k$ is a rational 
function with one and only one pole at $1/\beta$ and the order of the pole is 
$\deg(S)+1$. Since $(\tilde y_k)_{k\ge 1}$ is bounded, \eqref{eq:ps} 
converges on the open unit disk and has at most simple poles on the unit 
circle. Thus $H_j=G_j$ for every $j$ such that $\lvert \alpha_j\rvert 
>1$ and $\deg(H_j-G_j)\le 0$ for every $j$ such that $\lvert \alpha_j\rvert=1$. 

Suppose that we are in case (a): $\lvert \alpha_1\rvert >1$, $F_1\neq 0$, and 
$\alpha_1$ is not an algebraic integer. Then $H_1=G_1$. Since $c_{1,f_1}\neq 
0$ by \eqref{eq:cj}, it follows that $G_1\neq 0$ and consequently $\alpha_1$ 
is an algebraic integer. Contradiction. 

Suppose that we are in case (b): $\lvert \alpha_1\rvert =1$, 
$f_1=\deg(F_1)\ge 1$, and $\alpha_1$ is not an algebraic integer. Then 
$\deg(H_1-G_1)\le 0$. Since $c_{1,f_1}\neq 0$ by \eqref{eq:cj}, it follows 
that $G_1\neq 0$ and consequently $\alpha_1$ is an algebraic integer. 
Contradiction. 

Suppose that we are in case (c): $\lvert \alpha_1\rvert>1$, $\lvert 
\alpha_j\rvert>1$, $\sigma(\alpha_1)=\alpha_j$, $\sigma(F_1)\neq F_j$. Then 
$H_1=G_1$, $H_j=G_j$. It follows that $\sigma(H_1)=H_j$. Let $i$ be the 
largest integer such that $\sigma(\xi_{1,i})\neq \xi_{j,i}$. Then 
$\sigma(c_{1,i})\neq c_{j,i}$ by \eqref{eq:cj}. Contradiction. 

Suppose that we are in case (d): $\lvert \alpha_1\rvert\ge 1$, $\lvert 
\alpha_j\rvert\ge 1$, $\sigma(\alpha_1)=\alpha_j$, $\deg(\sigma(F_1)- F_j)\ge 
1$. Then $\deg(H_1-G_1)\le 0$, $\deg(H_j-G_j)\le 0$, and consequently  
$\deg(\sigma(H_1)-H_j)\le 0$. Let $i\ge 1$ be the largest integer such that 
$\sigma(\xi_{1,i})\neq \xi_{j,i}$. Then $\sigma(c_{1,i})\neq c_{j,i}$ by 
\eqref{eq:cj}. Contradiction. 

It remains to prove the last assertion of Proposition \ref{p:recur}. Without 
changing $R$, the above still holds for $k\ge k_0$. We are under condition 
(e). In particular, $\lvert \alpha_1\rvert >1$, $\lvert \alpha_j\rvert \ge 
1$, $\sigma(\alpha_1)=\alpha_j$. Then $H_1=G_1$. Moreover, using the argument in case (d), 
we have $\deg(\sigma(F_1)- F_j)\le 0$. That is, $\sigma(\xi_{1,i})=\xi_{j,i}$ 
for all $i\ge 1$. Thus, 
\[H_j-G_j=H_j-\sigma(H_1)=Q(\alpha_j)(\xi_{j,0}-\sigma(\xi_{1,0}))=Q(\alpha_j)(F_j-\sigma(F_1))\] 
by \eqref{eq:cj}. Write 
$S(x):=(P(x)/(x-\alpha_j))^{f+1}=\sum_{i=0}^{s}w_ix^i$. By \eqref{eq:yk}, for 
$k\ge k_0$, we have 
\[\sum_{i=0}^{s} w_i \tilde y_{k+i} =\sum_{i=0}^s w_i (H_j-G_j)\alpha_j^{k+i}=\alpha_j^k S(\alpha_j)Q(\alpha_j)(F_j-\sigma(F_1)).\]
Thus $L(SQ)\cdot \limsup_{k\to \infty} \lvert y_k\rvert \ge \lvert 
S(\alpha_j)Q(\alpha_j)\rvert \cdot \lvert F_j-\sigma(F_1)\rvert$. It suffices 
to take $N=\min_\alpha \lvert R_{\alpha}(\alpha)\rvert/L(R_\alpha)$, where 
$\alpha$ runs through roots of $R$, $R_\alpha(x)=R(x)/(x-\alpha)^{n_\alpha}$, 
and $n_\alpha$ denotes the order of $R(x)$ at $x=\alpha$. 
\end{proof}

The following lemma is a consequence of the Morse--Hedlund theorem on subword 
complexity \cite[Theorem 7.3]{MH}. The case $M=1$ of the lemma is 
\cite[Lemma~2]{Dubickas} or \cite[Corollary A.4]{Bugeaud-2012}. We say that a 
word $v_1\dots v_e$ occurs in a word $w_1w_2\dots$ at position $n$ if 
$w_{n+i-1}=v_i$ for all $1\le i\le e$. 

\begin{lemma}\label{l:Morse}
Let $W=w_1w_2\dots$ be an infinite word over a finite alphabet $S$ that is 
not ultimately periodic. Let $e\ge 0$ and $M\ge 1$ be integers. Then there 
exist words $U,V$ of length $e$, letters $s,s',t,t'\in S$, and congruence 
classes $\bar m, \bar n$ modulo $M$ such that $s\neq s'$, $t\neq t'$, and the 
words $sU$ and $s'U$ both occur in $W$ at infinitely many positions in $\bar 
m$, and the words $Vt$ and $Vt'$ both occur in $W$ at infinitely many 
positions in $\bar n$. 
\end{lemma}

\begin{proof}
The case $e=0$ follows from the assumption that $W$ is not ultimately 
periodic. Assume $e\ge 1$. We lift $W$ to a word $\tilde W=\tilde w_1\tilde 
w_2\dots$ over the alphabet $A=S\times (\Z/M\Z)$, where $\tilde w_i=(w_i,\bar 
i)$. Here $\bar i$ denotes the congruence class of $i$ modulo $M$. Then 
$\tilde W$ is not ultimately periodic. Applying \cite[Corollary 
A.4]{Bugeaud-2012} to $\tilde W$, we get words $\tilde U$, $\tilde V$ of 
length $e$ over $A$ and letters $\tilde s,\tilde s',\tilde t,\tilde t'\in A$ 
such that $\tilde s\neq \tilde s'$, $\tilde t\neq \tilde t'$, and each of the 
words $\tilde s\tilde U$, $\tilde s'\tilde U$, $\tilde V\tilde t$, $\tilde 
V\tilde t'$ occur in $\tilde W$ at infinitely many positions. Write $\tilde 
s=(s,\bar m)$,  $\tilde s'=(s',\bar m')$, $\tilde t=(t,\bar l)$, $\tilde 
t'=(t',\bar l')$. Since $\tilde s\tilde U$ and $\tilde s'\tilde U$ occur both 
in $\tilde W$, we have $\bar m=\bar m'$ and consequently $s\neq s'$. 
Similarly, $\bar l=\bar l'$ and $t\neq t'$. It then suffices to take $\bar 
n=\bar l-\bar e$ and $U$, $V$ to the words obtained from $\tilde U$, $\tilde 
V$ by applying the projection $A\to S$.  
\end{proof}

The following lemma will be used in the proof of Theorem \ref{p:Dubickas}. 
The case $M=1$ and $s_k=\tilde y_k$ of the lemma is the second part of 
\cite[Lemma~3]{Dubickas}. Following \cite[Definition~1.9]{Astorg-Boc 
Thaler-2024}, we say that a sequence $(x_{k})_{k \geqslant 1}$ 
\emph{converges to a cycle} if there exists an integer $M\ge 1$ such that the 
subsequence $(x_{kM + i})_{k \geqslant 1}$ converges for each $0 \leqslant i 
< M$.

\begin{lemma}\label{l:Dubickas}
Let $P(x)=\sum_{i=0}^d a_i x^i\in \R[x]$ be a nonzero polynomial and let 
$(y_k)_{k\ge 1}$ be a bounded sequence of real numbers.  Let $\tilde 
y_k=\sum_{i=0}^d a_i y_{k+i}$ for every $k\ge 1$. Assume there exists a 
sequence $(s_k)_{k\ge 1}$ taking values in a finite set $S$ of real numbers 
such that $(s_k)_{k\ge 1}$ is not ultimately periodic and $(\tilde 
y_k-s_k)_{k\ge 1}$ converges to a cycle. Let $M\ge 1$ be an integer and 
$\delta=\min_{\substack{k\equiv k'\pmod{M}\\s_k\neq s_{k'}}} \lvert 
s_k-s_{k'}\rvert$. Then there exists an integer $l$ such that 
\[\limsup_{k\to \infty} y_{kM+l} -\liminf_{k\to \infty} y_{kM+l}\ge \delta/\ell(P).\]
\end{lemma}

\begin{proof}
We refine Dubickas' proof of \cite[Lemma~3]{Dubickas} using words with the 
help of Lemma \ref{l:Morse}. Let $t_k:=\tilde y_k-s_k$. Up to replacing $M$ 
by a multiple, we may assume that $(t_{kM+i})_{k\ge 1}$ converges for each 
$0\le i<M$. Let $\mu_{\bar i}=\limsup_{k\to \infty} y_{kM+i}$ and $\nu_{\bar 
i}=\liminf_{k\to \infty} y_{kM+i}$, which only depend on the congruence class 
$\bar i$ of $i$ modulo $M$. Let $\epsilon>0$. There exists a polynomial $Q\in 
\R[x]$ of leading coefficient $1$ or constant coefficient $1$ such that 
$L(PQ)<\ell(P)+\epsilon$. We will prove the case of constant coefficient $1$. 
The proof in the other case is completely parallel. Write $Q(x)=\sum_{i=0}^e 
b_i x^i$, where $b_0=1$. There exists $N\ge 1$ such that for all $k,k'\ge N$ 
satisfying $k\equiv k'\pmod M$, we have $y_k\in (\nu_{\bar 
k}-\epsilon,\mu_{\bar k}+\epsilon)$ and $\lvert t_k-t_{k'}\rvert 
<\epsilon/L(Q)$. We consider the infinite word $W=s_1s_2\dotsm$ over the 
alphabet $S$. We let $W_{k,k+e}=s_k\dotsm s_{k+e}$ denote the subword of $W$ 
of length $e+1$. By Lemma \ref{l:Morse}, there exist a word $V=v_1\dotsm v_e$ 
over $S$ of length $e$, letters $u_1,u_2\in S$, and a congruence class $\bar 
n$ modulo $M$ such that $u_1\neq u_2$ and the sets $K_j=\{k\in \bar n\mid 
W_{k,k+e}=u_jV\}$ $(j=1,2)$ are infinite. We may assume $u_1>u_2$. Choose a 
representative $n\ge N$ of $\bar n$. Let $c=\sum_{i=1}^e b_iv_i$ and 
$c'=\sum_{i=0}^e b_it_{n+i}$. For $k\in K_j$, $j=1,2$, we have $\sum_{i=0}^e 
b_i s_{k+i}=u_j+c$. For $k$ in $\bar n$ satisfying $k\ge N$, we have $\lvert 
\sum_{i=0}^e b_i t_{k+i}-c'\rvert <\epsilon$. Write $PQ(x)=\sum_{i=0}^{d+e} 
r_i x^i$. Let $I_+$ and $I_-$ denote the set of indices $i$ such that $r_i>0$ 
and $r_i<0$, respectively. Choosing any $k\in K_1$ satisfying $k\ge N$, we 
get 
\begin{equation}\label{eq:1}
\sum_{i\in I_+} r_i(\mu_{\overline{i+n}}+\epsilon)- 
\sum_{i\in I_-} r_i(\nu_{\overline{i+n}}-\epsilon)>u_1+c+c'-\epsilon.
\end{equation} 
Choosing any $k\in K_2$ satisfying $k\ge N$, we get 
\begin{equation}\label{eq:2} 
\sum_{i\in I_+} r_i(\nu_{\overline{i+n}}-\epsilon)- 
\sum_{i\in I_-} r_i(\mu_{\overline{i+n}}+\epsilon)<u_2+c+c'+\epsilon.
\end{equation}
Subtracting \eqref{eq:2} from \eqref{eq:1}, we get
\[(\ell(P)+\epsilon)(\Delta+2\epsilon)>L(PQ)(\Delta+2\epsilon) > u_1-u_2-2\epsilon\ge \delta-2\epsilon,\]
where $\Delta=\max_{\bar i}(\mu_{\bar i}-\nu_{\bar i})$. Since $\epsilon>0$ 
is arbitrary, we get $\ell(P)\Delta\ge \delta$. 
\end{proof}

\begin{proof}[Proof of Theorem \ref{p:Dubickas}]
For \eqref{eq:pD1} it suffices to show that for every decomposition $R=PQ$ in 
$\Z[x]$ with $P=R/Q$ as in Definition \ref{d:reduced}, we have 
\begin{equation}\label{eq:pD}
\max_{0\le l<M} \left(\limsup_{k\to \infty}y_{kM+l}-\liminf_{k\to \infty}y_{kM+l}\right)\ge 1/\ell(P).
\end{equation}
Write $P(x)=\sum_{i=0}^d a_i x^i$ and $Q(x)=\sum_{i=0}^e b_ix^i$. Let 
\[t_k=\sum_{i=0}^d 
a_ix_{k+i},\quad\tilde y_k=\sum_{i=0}^d 
a_iy_{k+i},\quad -s_k=\sum_{i=0}^d a_i z_{k+i}.
\] 
Then $\sum_{i=0}^{e}b_it_{k+i}=0$ for all $k\ge 1$. By the assumption about 
the roots of $Q$, it follows that $(t_k)_{k\ge 1}$ converges to a cycle of 
period $m\ge 1$, where $m$ is the least common multiple of the orders of the 
roots of unity that are roots of $Q$. (By convention $m=1$ if there are no 
such roots of unity.) Up to replacing $M$ by a multiple, we may assume that 
$m$ divides $M$. Then $t_k=\tilde y_k -s_k$ and $(\tilde y_k)_{k\ge 1}$ is 
bounded. Thus $(s_k)_{k\ge 1}$ is a bounded sequence of integers. Since 
$(z_k)_{k\ge 1}$ is not a linear recurrent sequence, neither is $(s_k)_{k\ge 
1}$. In particular, $(s_k)_{k\ge 1}$ is not ultimately periodic. This implies 
\eqref{eq:pD} by Lemma \ref{l:Dubickas}. 

Let
\[\hat y_k=\sum_{i=0}^r 
u_iy_{k+i},\quad \hat z_k=\sum_{i=0}^r u_i z_{k+i}.
\] 
Then $\hat y_k+\hat z_k=\sum_{i=0}^r u_ix_{k+i}=0$. Thus $\hat y_k=-\hat 
z_k\in \Z$ and $(\hat y_k)_{k\ge 1}$ is not a linear recurrent sequence. In 
particular, for every $N\ge 1$, there exists $n\ge N$ such that 
\[1\le \lvert \hat y_n\rvert \le L(R)\sup_{k\ge n}\lvert y_k\rvert.\] 
Letting $N\to \infty$, we get 
\[1\le L(R)\limsup_{k\to \infty} \lvert 
y_k\rvert.\qedhere\] 
\end{proof}

\begin{proof}[Proof of Theorem \ref{t:Dubickas}]
The inequality \eqref{eq:Dubickas2} follows from Proposition \ref{p:recur} 
and Theorem \ref{p:Dubickas} applied to the decomposition 
$x_k=(\{x_k+1/2\}-1/2)+\lfloor x_k+1/2\rfloor$. Indeed, $\lVert 
x\rVert=\lvert \{x+1/2\}-1/2\rvert$. Assume that for each congruence class 
$\bar l$ modulo $M$, there exists an interval $I_{\bar l} $ of $\R/\Z$ of 
length $<1/\lambda(R)\le 1$ such that the set of limit values of 
$(\pi(x_{kM+l}))_{k\ge 1}$ is contained in $I_{\bar l}$. For each $\bar l$, 
choose $\eta_{\bar l}\in \R$ such that $\pi(\eta_{\bar l})\notin I_{\bar l}$. 
Proposition \ref{p:recur} and Theorem \ref{p:Dubickas} applied to the 
decomposition $x_k=(\{x_k-\eta_{\bar k}\}+\eta_{\bar k})+\lfloor 
x_k-\eta_{\bar k}\rfloor$ then yield a contradiction. 

It remains to show that $E$ is an infinite set. Suppose that $E$ is finite. 
Let $\epsilon=1/L(R)$. If we have $i$, $j$, $\sigma$ in case (e), then, up to 
multiplying $x_k$ by a positive integer, we may assume that $N\cdot \lvert 
\sigma(F_i)-F_j\rvert>1$, where $N>0$ is the constant in Proposition 
\ref{p:recur}. In all cases, by Kronecker's approximation theorem 
\cite[Theorem 1.18]{Bugeaud-2012}, up to multiplying $x_k$ by a positive 
integer, we may assume that $E$ is contained in 
$\pi((-\epsilon/2,\epsilon/2))$. Then there exists $k_0\ge 1$ such that for 
all $k\ge k_0$, $y_k:=\{x_k+1/2\}-1/2$ satisfies $\lvert y_k\rvert=\lVert 
x_k\rVert<\epsilon/2$. Let $z_k:=\lfloor x_k+1/2\rfloor$, 
\[\tilde y_k:=\sum_{m=0}^r u_m y_{k+m},\quad \tilde z_k:=\sum_{m=0}^r 
u_mz_{k+m}.
\] 
Then $\tilde y_k+\tilde z_k=\sum_{m=0}^r u_mx_{k+m}=0$. Thus $\tilde 
y_k=-\tilde z_k\in \Z$ and, for $k\ge k_0$, $\lvert \tilde y_k\rvert< 
L(R)\epsilon/2=1/2$, which forces $\tilde y_k=0$ and consequently $\tilde 
z_k=0$, contradicting Proposition \ref{p:recur}. (In case (e), Proposition 
\ref{p:recur} implies that $\limsup_{k\to \infty}\lvert y_k\rvert >1$, which 
contradicts the trivial bound $\lvert y_k\rvert\le 1/2$.) 
\end{proof} 

\begin{proof}[Proof of Corollary \ref{c:iff}]
The necessity of the conditions follows from Theorem \ref{t:Dubickas}. 
Indeed, if (1) does not hold, then (a) holds. If (2) does not hold, then (e) 
holds. If (3) does not hold, then (d) holds. (In the last two cases, if 
$\sigma(\alpha_i)$ does not appear in $\alpha_1,\dots,\alpha_n$, we take 
$\alpha_{n+1}=\sigma(\alpha_i)$ and $F_{n+1}=0$ in Theorem \ref{t:Dubickas}.) 

Conversely, assume that (1), (2), and (3) are satisfied. Note that if all 
Galois conjugates of an algebraic number $\alpha\in \CC$ have modulus $\le 1$ 
and one of them, say $\beta$, has modulus~$1$, then all of them have modulus 
$1$. Indeed, since $\bar\beta=\beta^{-1}$ is a Galois conjugate of~$\beta$, 
the minimal polynomial $P(x)\in \Q[x]$ of $\alpha$ is self-reciprocal, i.e. 
$x^{\deg(P)}P(1/x)=P(0)P(x)$. Thus the set of Galois conjugates of $\alpha$ 
is stable under the map $x\mapsto 1/x$. It follows that every algebraic 
number $\alpha\in \CC$ falls into one of the following three mutually 
exclusive categories: 
\begin{enumerate}[(i)]
\item $\alpha$ has a Galois conjugate of modulus $>1$; or
\item All Galois conjugates of $\alpha$ have modulus $1$; or
\item All Galois conjugates of $\alpha$ have modulus $<1$.
\end{enumerate}
Consequently conditions (1), (2), (3) imply that $(x_k)_{k\ge 1}$ is a finite 
sum of sequences of the following types: 
\begin{enumerate}[(A)]
\item $(\sum_\sigma \sigma (F(k)\alpha^k))_{k\ge 1}$, where $\alpha\in \CC$ 
    is an algebraic integer, $F(x)\in \Q(\alpha)[x]$, and $\sigma$ runs 
    through embeddings $\Q(\alpha)\to \CC$; 
\item $(\xi\zeta^k)_{k\ge 1}$, where $\xi\in \CC$ and $\zeta$ is a root of 
    unity; 
\item $(F(k)\alpha^k)_{k\ge 1}$, where $F(x)\in \CC[x]$ and $\lvert 
    \alpha\rvert <1$. 
\end{enumerate}
It suffices to show that for $(x_k)_{k\ge 1}$ of one of the three types, the 
set $E\subseteq \CC/(\Z[\sqrt{-1}])$ of limit values of 
$(\pi_{\CC}(x_k))_{k\ge 1}$ is finite, where $\pi_{\CC}\colon \CC\to 
\CC/(\Z[\sqrt{-1}])$ is the projection. The finiteness of $E$ is trivial for 
types (B) and (C). For type (A), up to multiplication by a positive integer, 
we may assume that $F(x)\in \Z[\alpha][x]$. Then $(x_k)_{k\ge 1}$ is a 
sequence of integers and the finiteness of $E$ is again trivial. 
\end{proof}

\begin{proof}[Proof of Corollary \ref{c:1}]
By Corollary \ref{c:iff}, $E$ is finite if and only if $\lvert\alpha\rvert$ 
is a Pisot number and $F(x)\in \Q(\alpha)[x]$ (conditions (1) and (2)) and 
$G(x)-G(0)\in \Q[x]$ (condition (3)). The last assertion follows from 
Proposition \ref{prop} and Remark \ref{r:Spec}, and Example \ref{ex:adm}(1). 
\end{proof}

\begin{proof}[Proof of Corollary \ref{c:2}]
The inequality \eqref{eq:c2} and the assertion on $(\pi(x_{kM+l}))_{k\ge 1}$ 
follow from Theorem \ref{t:Dubickas} applied to 
$R(x)=(x-1)^{g+1}P(x)^{f+1}Q(x)$ and to $R(x)=(x-1)^{g+1}P(x)^{f+1}$, 
respectively. Indeed, case (1) follows from cases (a) and (c), case (2) 
follows from case (c$'$), case (3) follows from case (d$'$), and case (4) 
follows from case (d). Note that
\[\lambda((x-1)^{g+1}P(x)^{f+1})\ge 
\ell((x-1)^{\max(g,0)}P(x)^{f+1}).\] 
The last assertion follows from 
Proposition \ref{prop}, Remark \ref{r:Spec}, and Example \ref{ex:adm}(2) 
applied to the linear recurrence relation defined by the coefficients of 
$P^{f+1}Q$. 
\end{proof}

The following corollary of Theorem \ref{t:Dubickas} immediately implies 
Theorem \ref{t:rational}. 

\begin{cor}
Let $\xi_1,\dots,\xi_n$ be nonzero real numbers and let 
$p_1,\dots,p_n,q_1,\dots,q_n$ be integers such that $p_i>\lvert q_i\rvert 
>0$ for all $1\le i\le n$ and
$\frac{p_1}{q_1},\dots,\frac{p_n}{q_n}$ are pairwise distinct. Assume that 
$\xi_1$ is irrational or $p_1/q_1\notin \Z$. Let 
$x_k=\sum_{i=1}^n\xi_i(p_i/q_i)^k$ for $k\ge 1$. Then the set of limit values 
of $(\pi(x_k))_{k\ge 1}$ is an infinite set and for every integer $M\ge 1$, 
there exists an integer $l\ge 0$ such that the set of limit values of 
$(\pi(x_{kM+l}))_{k\ge 1}$ is not contained in any interval of $\R/\Z$ of 
length $<\frac{1}{p_1\dotsm p_n}$. 
\end{cor}

\begin{proof}
This follows from cases (a) (if $p_1/q_1\notin \Z$) and (c$'$) (if $\xi_1$ is 
irrational) of Theorem \ref{t:Dubickas} applied to 
$R(x)=\prod_{j=1}^n(q_jx-p_j)$. Since $R(x)$ divides 
$R_m(x)=\prod_{j=1}^n(\frac{q_j^m}{p_j^{m-1}}x^m-p_j)$ in $\R[x]$ for all 
$m\ge 1$ and $\lim_{m\to \infty} L(R_m)= p_1\dotsm p_n$, we have 
$\lambda(R)\le \ell(R)\leqslant \inf_{m\geqslant1}L(R_m)=p_1\dotsm p_n$. 
\end{proof}

\begin{remark}
In the proof above we have in fact $\ell(R)=p_1\dotsm p_n$. Indeed, by 
\cite[Proposition (ii)]{Schinzel}, $\ell(R)\ge M(R)=p_1\dotsm p_n$, where 
$M(R)$ denotes the Mahler measure of~$R$. 
\end{remark}

\section{Discussion}\label{s:final}

The {\em spectrum} of a sequence $(x_k)_{k\geqslant 1}$ of real numbers is 
the set of $\theta\in\R/\Z$ such that the sequence 
$(\pi(x_k)-k\theta)_{k\geqslant 1}$ is \emph{not} uniformly distributed on 
$\R/\Z$. We let ${\sf Spec}(x_\bullet)$ denote the spectrum of 
$(x_k)_{k\geqslant 1}$. Mend\`es France \cite[p.~7-05, 
Problem~6*]{France-1973} asked the following question (also mentioned in 
\cite[Problem 10.4]{Bugeaud-2012}): For real numbers $\xi\neq 0$ and 
$\alpha>1$, is the spectrum of the sequence $(\xi \alpha^k)_{k\geqslant 1}$ 
at most countable? 

While this question remains open, we have a better understanding for the 
following variants of the notion of spectrum. 

\begin{definition}
For a sequence of real numbers $(x_k)_{k\geqslant 1}$, define
\[
\aligned
{\sf Spec_{fin}}(x_\bullet)&:=\bigl\{\theta\in\R/\Z\mid(\pi(x_k)-k\theta)_{k\geqslant 1}~\text{has finitely many limit values}\bigr\},\\
{\sf Spec_{count}}(x_\bullet)&:=\bigl\{\theta\in\R/\Z\mid(\pi(x_k)-k\theta)_{k\geqslant 1}~\text{has countably many limit values}\bigr\}.
\endaligned
\]
\end{definition}

Obviously ${\sf Spec_{fin}}(x_\bullet)\subseteq{\sf 
Spec_{count}}(x_\bullet)\subseteq{\sf Spec}(x_\bullet)$.

With this notation, the case $G(x)=-\theta x$ of Corollary \ref{c:1} can be 
stated as follows. 

\begin{cor}
Let $\alpha$ be a real algebraic number with $\lvert\alpha\rvert>1$ and let 
$F(x)\in \R[x]$ be nonzero. Let $x_k=F(k)\alpha^k$. Then 
\[{\sf Spec_{fin}}(x_\bullet)=\begin{cases}
\Q/\Z & \text{$\lvert \alpha\rvert$ is a Pisot number and $F(x)\in \Q(\alpha)[x]$}\\
\emptyset & \text{otherwise.}  
\end{cases}
\]
Moreover, ${\sf Spec_{count}}(x_\bullet)\subseteq\Q/\Z$.
\end{cor}

\subsection*{Acknowledgements} We thank Shaoshi Chen for useful discussions. 

Zhangchi Chen was supported by the National Key R\&D Program of China (No.\ 
2025YFA1018300), the National Natural Science Foundation of China (No.\ 
12501104), the Science and Technology Commission of Shanghai Municipality 
(No.\ 22DZ2229014), the Shanghai Sailing Program (No.\ 24YF2709900), and the 
Shanghai Pujiang Program (No.\ 24PJA023). Weizhe Zheng was supported by the 
National Natural Science Foundation of China (grant numbers 12125107, 
12288201) and the Chinese Academy of Sciences Project for Young Scientists in 
Basic Research (grant number YSBR-033).


\begin{bibdiv}
\begin{biblist}
\bib{Astorg-Boc Thaler-2024}{article}{
   author={Astorg, Matthieu},
   author={Boc Thaler, Luka},
   title={Dynamics of skew-products tangent to the identity},
   journal={J. Eur. Math. Soc. (JEMS)},
   volume={28},
   date={2026},
   number={2},
   pages={559--618},
   issn={1435-9855},
   review={\MR{5031425}},
   doi={10.4171/jems/1566},
}

\bib{Bugeaud-2012}{book}{
   author={Bugeaud, Yann},
   title={Distribution modulo one and Diophantine approximation},
   series={Cambridge Tracts in Mathematics},
   volume={193},
   publisher={Cambridge University Press, Cambridge},
   date={2012},
   pages={xvi+300},
   isbn={978-0-521-11169-0},
   review={\MR{2953186}},
   doi={10.1017/CBO9781139017732},
}

\bib{CYZ}{article}{ 
   author={Chen, Zhangchi},
   author={Ye, Zihao},
   author={Zheng, Weizhe},
   title={On a question of {A}storg and {B}oc {T}haler},
   date={2026},
   note={arXiv:2511.21324},
} 

\bib{Dubickas-BLMS}{article}{
   author={Dubickas, Art\=uras},
   title={Arithmetical properties of powers of algebraic numbers},
   journal={Bull. London Math. Soc.},
   volume={38},
   date={2006},
   number={1},
   pages={70--80},
   issn={0024-6093},
   review={\MR{2201605}},
   doi={10.1112/S0024609305017728},
}

\bib{Dubickas}{article}{
   author={Dubickas, Art\=uras},
   title={On the distance from a rational power to the nearest integer},
   journal={J. Number Theory},
   volume={117},
   date={2006},
   number={1},
   pages={222--239},
   issn={0022-314X},
   review={\MR{2204744}},
   doi={10.1016/j.jnt.2005.07.004},
}

\bib{Dubickas-2007}{article}{
   author={Dubickas, Art\=uras},
   title={Arithmetical properties of linear recurrent sequences},
   journal={J. Number Theory},
   volume={122},
   date={2007},
   number={1},
   pages={142--150},
   issn={0022-314X},
   review={\MR{2287116}},
   doi={10.1016/j.jnt.2006.04.002},
}

\bib{DN}{article}{
   author={Dubickas, Art\=uras},
   author={Novikas, Aivaras},
   title={Integer parts of powers of rational numbers},
   journal={Math. Z.},
   volume={251},
   date={2005},
   number={3},
   pages={635--648},
   issn={0025-5874},
   review={\MR{2190349}},
   doi={10.1007/s00209-005-0827-4},
}

\bib{FLP}{article}{
   author={Flatto, Leopold},
   author={Lagarias, Jeffrey C.},
   author={Pollington, Andrew D.},
   title={On the range of fractional parts $\{\xi(p/q)^n\}$},
   journal={Acta Arith.},
   volume={70},
   date={1995},
   number={2},
   pages={125--147},
   issn={0065-1036},
   review={\MR{1322557}},
   doi={10.4064/aa-70-2-125-147},
}


\bib{GTT}{article}{
   author={Grekos, Georges},
   author={Toma, Vladim\'ir},
   author={Tomanov\'a, Jana},
   title={A note on uniform or Banach density},
   journal={Ann. Math. Blaise Pascal},
   volume={17},
   date={2010},
   number={1},
   pages={153--163},
   issn={1259-1734},
   review={\MR{2674656}},
}

\bib{Hardy}{article}{
   author={Hardy, Godfrey Harold},
   title={A problem of Diophantine approximation},
   journal={J. Indian Math. Soc.},
   volume={11},
   date={1919},
   pages={162--166},
}

\bib{Hlawka0}{article}{
   author={Hlawka, Edmund},
   title={Zur formalen Theorie der Gleichverteilung in kompakten Gruppen},
   language={German},
   journal={Rend. Circ. Mat. Palermo (2)},
   volume={4},
   date={1955},
   pages={33--47},
   issn={0009-725X},
   review={\MR{0074489}},
   doi={10.1007/BF02846027},
}

\bib{Hlawka}{article}{
   author={Hlawka, Edmund},
   title={Erbliche Eigenschaften in der Theorie der Gleichverteilung},
   language={German},
   journal={Publ. Math. Debrecen},
   volume={7},
   date={1960},
   pages={181--186},
   issn={0033-3883},
   review={\MR{0125103}},
   doi={10.5486/pmd.1960.7.1-4.16},
}

\bib{Lawton}{article}{
   author={Lawton, B.},
   title={A note on well distributed sequences},
   journal={Proc. Amer. Math. Soc.},
   volume={10},
   date={1959},
   pages={891--893},
   issn={0002-9939},
   review={\MR{0109818}},
   doi={10.2307/2033616},
}

\bib{Mahler}{article}{
   author={Mahler, K.},
   title={An unsolved problem on the powers of $3/2$},
   journal={J. Austral. Math. Soc.},
   volume={8},
   date={1968},
   pages={313--321},
   review={\MR{0227109}},
}

\bib{France-1973}{article}{
   author={Mend\`es France, Michel},
   title={Les ensembles de B\'esineau},
   language={French},
   conference={
      title={S\'eminaire Delange-Pisot-Poitou (15\`eme ann\'ee: 1973/74),
      Th\'eorie des nombres, Fasc. 1},
   },
   book={
      publisher={Secr\'etariat Math., Paris},
   },
   date={1975},
   pages={Exp. No. 7, 1--6},
   review={\MR{0412139}},
}


\bib{MH}{article}{
   author={Morse, Marston},
   author={Hedlund, Gustav A.},
   title={Symbolic Dynamics},
   journal={Amer. J. Math.},
   volume={60},
   date={1938},
   number={4},
   pages={815--866},
   issn={0002-9327},
   review={\MR{1507944}},
   doi={10.2307/2371264},
}

\bib{Petersen}{article}{
   author={Petersen, G. M.},
   title={`Almost convergence'\ and uniformly distributed sequences},
   journal={Quart. J. Math. Oxford Ser. (2)},
   volume={7},
   date={1956},
   pages={188--191},
   issn={0033-5606},
   review={\MR{0095812}},
   doi={10.1093/qmath/7.1.188},
}

\bib{Pisot}{article}{
   author={Pisot, Charles},
   title={La r\'epartition modulo 1 et les nombres alg\'ebriques},
   language={French},
   journal={Ann. Scuola Norm. Super. Pisa Cl. Sci. (2)},
   volume={7},
   date={1938},
   number={3-4},
   pages={205--248},
   issn={0391-173X},
   review={\MR{1556807}},
}

\bib{Pisot-1946}{article}{
   author={Pisot, Charles},
   title={R\'epartition $({\rm mod} 1)$ des puissances successives des
   nombres r\'eels},
   language={French},
   journal={Comment. Math. Helv.},
   volume={19},
   date={1946},
   pages={153--160},
   issn={0010-2571},
   review={\MR{0017744}},
   doi={10.1007/BF02565954},
}

\bib{Schinzel}{article}{
   author={Schinzel, Andrzej},
   title={On the reduced length of a polynomial with real coefficients},
   journal={Funct. Approx. Comment. Math.},
   volume={35},
   date={2006},
   pages={271--306},
   issn={0208-6573},
   review={\MR{2271619}},
   doi={10.7169/facm/1229442629},
}


\bib{Vija}{article}{
   author={Vijayaraghavan, Tirukkannapuram},
   title={On the fractional parts of the powers of a number. II},
   journal={Proc. Cambridge Philos. Soc.},
   volume={37},
   date={1941},
   pages={349--357},
   issn={0008-1981},
   review={\MR{0006217}},
}

\bib{Weyl}{article}{
   author={Weyl, Hermann},
   title={\"Uber die Gleichverteilung von Zahlen mod. Eins},
   language={German},
   journal={Math. Ann.},
   volume={77},
   date={1916},
   number={3},
   pages={313--352},
   issn={0025-5831},
   review={\MR{1511862}},
   doi={10.1007/BF01475864},
}
\end{biblist}
\end{bibdiv}
\end{document}